\documentclass[12pt]{article}
\usepackage{amsfonts,amssymb,amsmath,url}
\usepackage{xcolor}

\usepackage{amsthm}

\usepackage{amssymb}
\usepackage{graphics}
\usepackage{tikz}
\usetikzlibrary{shapes,backgrounds,calc}
\usepackage{latexsym}
\usepackage{multicol}
\usepackage{verbatim,enumerate}
\usepackage{accents}
\usepackage{cite}

\usepackage[colorlinks=true, linkcolor=blue, citecolor=blue, urlcolor=blue]{hyperref}
\usepackage{hyperref}
\usepackage{amsmath, amscd,url}

\usepackage{setspace}

\usepackage{pstricks}
\allowdisplaybreaks

\advance\textwidth by 1.3in \advance\oddsidemargin by -.6in \advance\evensidemargin by -.6in
\newtheorem{theorem}{Theorem}[section]

\newtheorem{lemma}[theorem]{Lemma}

\newtheorem{rem}{Remark}
\newtheorem{defn}{Definition}

\newcommand{\grtype}{\mathop{\mbox{$\mathrm{A}$}}}
\newcommand{\partype}{\mathop{\mbox{$\mathrm{B}$}}}
\newcommand{\spec}{\mathop{\mathrm{Spec}}}

\newcommand{\adj}{\mathop{\mathrm{Adj}}}
 
\begin{document}

\title{Main functions and the spectrum of super graphs}
\author{G. Arunkumar\\
{\small Department of Mathematics,  Indian Institute of Technology
Madras, Chennai - 600036, India
}\\
{\small Email: arun.maths123@gmail.com}\\
Peter J. Cameron\footnote{Corresponding author}\\
{\small School of Mathematics and Statistics, University of St Andrews, Fife
KY16 9SS, UK}\\
{\small Email: pjc20@st-andrews.ac.uk}\\
R. Ganeshbabu\\
{\small Department of Mathematics,  Indian Institute of Technology
Madras, Chennai - 600036, India}\\
{\small Email: ma22d011@smail.iitm.ac.in}\\
Rajat Kanti Nath\\
{\small Department of Mathematical Sciences, Tezpur University, Sonitpur, Assam 784028, India.}\\
{\small Email: rajatkantinath@yahoo.com}
}
\date{}
\maketitle

\begin{abstract}
Let $\grtype$ be a graph type and $\partype$ an equivalence relation on a group $G$. Let $[g]$ be the equivalence class of $g$ with respect to the equivalence relation $\partype$. The $\partype$ super$\grtype$ graph of $G$ is an undirected graph whose vertex set is $G$ and two distinct  vertices $g, h \in G$ are adjacent  if $[g] = [h]$ or there exist $x \in [g]$ and $y \in [h]$ such that $x$ and $y$ are adjacent in the $\grtype$ graph of $G$.
In this paper, we compute  spectrum of equality/conjugacy supercommuting graphs of dihedral/dicyclic groups and show that these graphs are not integral.  
\end{abstract}



 \section{Introduction}

This is a continuation of our work on Super graphs on groups done in \cite{acns} and \cite{part2}. In these earlier works, we introduced and studied three types of graphs,
and three equivalence relations defined on a group, viz.\ the power graph,
enhanced power graph, commuting graph, and the relations of equality,
conjugacy, and same order; for each choice of a graph type $\grtype$ and
an equivalence relation $\partype$, there is a graph, the \emph{$\partype$ 
super$\grtype$ graph} defined on $G$. Two distinct vertices $g, h \in G$ are adjacent in $\partype$ 
super$\grtype$ graph if $[g] = [h]$ or there exist $x \in [g]$ and $y \in [h]$ such that $x$ and $y$ are adjacent in the $\grtype$ graph of $G$, where $[g]$ is the equivalence class of $g$ with respect to the equivalence relation $\partype$.
For other terminologies and notations regarding graphs and supergraphs related to groups, we refer \cite{acns} and \cite{gong}. In the present work, we shall study the spectra of \emph{$\partype$ 
super$\grtype$ graphs} defined on a group $G$. The study of spectra of graphs associated with various algebraic structures is a widely explored area. For instance,  see \cite{DN-2017,BN2021,powergraphs} for  commuting graphs, commuting conjugacy class graphs and power graphs of finite groups; \cite{lovasz,paul} for Cayley graphs of finite groups; \cite{katja} for  zero-divisor graphs of finite commutative rings; \cite{FNS-2020} for commuting graphs of finite rings; \cite{fixing} for  1-point fixing graph etc. It is worth mentioning that Dalal et al. \cite{DMP-2023} also computed the spectrum and Laplacian spectrum of order/conjugacy supercommuting graphs of certain groups.

In \cite{part2}, it was observed that super graphs  are generalized compositions of complete graphs. Thus, to study the spectra of a \emph{$\partype$ 
super$\grtype$ graph},  results on  the spectra of generalized compositions of complete graphs are useful.
In \cite{cfmr}, Cardoso et al. express the adjacency spectra of a generalized composition of regular graphs in terms of the adjacency spectra of the factor graphs and the determinant of a quotient matrix. Despite this result, explicit computations of the spectra of generalized compositions, especially in the quotient matrix, tend to get difficult.

One important concept in the study of spectra is that of a \emph{main} eigenvalue of a graph, introduced by Cvetkovi\'c in \cite{cvet1}. An eigenvalue $\lambda$ of a graph $\Gamma$ is a main eigenvalue if it has an eigenvector which is not orthogonal to the all-one vector $\mathbf{1}_n$, where $n={\lvert V(\Gamma) \rvert}$, where $V(\Gamma)$ is the set of vertices of $\Gamma$. For a survey on the main eigenvalues of a graph, see \cite{rowlinson}. Closely related to main eigenvalues of a graph $\Gamma$ is the function $\mathbf{1}_{n}^t (\lambda I_n-A(\Gamma))^{-1}\mathbf{1}_{n}$ introduced as the \emph{coronal} by Mcleman and Mcnicholas in \cite{leman}. The coronal plays a crucial role in determining the spectra of graphs arising from various graph operations\cite{ref1,ref2,ref3,ref4}. In \cite{sma21}, Saravanan et al. generalize the coronal of a graph to the concept of a \emph{main} function: The main function associated with an $n \times n$ matrix $M$, corresponding to two $n \times 1$ vectors $u, v$ is the function $v^t (\lambda I_{n} -M )^{-1}u$. They proved that the spectra of an arbitrary generalized composition are determined completely by the spectra of the factor graphs and the associated main functions\cite[Theorem 4]{sma21}. These main functions pave the way for the effective computation of the spectra of interest in this paper. 

In Section \ref{shjoin}, we recall results related to spectrum of generalized composition of regular graphs and main functions.   
In Section \ref{sspec}, using main function techniques, we explicitly compute the adjacency spectrum of the class of equality supercommuting  and conjugacy supercommuting graphs of dihedral groups and dicyclic groups. 
It can be seen that these graphs are not integral.

\section{Notation and auxiliary results}\label{shjoin}
Let $\mathcal{H}$ be a graph and $V(\mathcal{H}) = \{1,\dots,k\}$. Let $\Gamma_1,\dots,\Gamma_k$ be a collection of graphs with  $V(\Gamma_i) = \{v_i^{1},\dots,v_i^{n_i}\}$ for $1 \le i \le k$. Then $\mathcal{H}$-join (also known as generalized composition) of the graphs $\Gamma_1, \dots, \Gamma_k$, denoted by $\mathcal{G} =:\mathcal{H}[\Gamma_1,\dots,\Gamma_k]$,  is a graph whose 
vertex set is $V(\Gamma_1) \sqcup \cdots \sqcup V(\Gamma_k)$ and two vertices $v_i^p$ and $v_j^q$ of $\mathcal{G}$ are adjacent if one the following conditions is satisfied:
\begin{enumerate}
	\item $i=j$, and $v_i^p$ and $v_j^q$ are adjacent vertices in $\Gamma_i$.
	\item $i \ne j$ and $i$ and $j$ are adjacent in $\mathcal{H}$. 
\end{enumerate}
This generalized composition of graphs introduced by  Schwenk \cite{sch} is also known as generalized lexicographic product and joined union (see \cite{accv,gerbaud,mwg,ww}). The induced subgraph on the union of the vertex sets of $\Gamma_i$ and $\Gamma_j$ is their disjoint union $\Gamma_1\sqcup\Gamma_2$ if $i\not\sim j$ in $\mathcal{H}$, or the join $\Gamma_1\vee\Gamma_2$ (with all edges between $V(\Gamma_1)$ and $V(\Gamma_2)$) if $i\sim j$ in $\mathcal{H}$.

Let $\mathcal{A}(\Gamma)$ be the adjacency matrix of $\Gamma$ and $\spec(\Gamma) := \spec(\mathcal{A}(\Gamma))$ be the spectrum of $\Gamma$, the multiset of eigenvalues of $\mathcal{A}(\Gamma)$. The following result of Cardoso \cite{cfmr}  gives the spectrum of $\mathcal{H}[\Gamma_1,\dots,\Gamma_k]$, to some extent.
\begin{theorem} \label{cardaso}\cite{cfmr} 
Let $\mathcal{G} = \mathcal{H}[\Gamma_1,\dots,\Gamma_k]$ where $V(\mathcal{H}) = \{1, \dots, k\}$, $|V(\Gamma_i)| = n_i$ and $\Gamma_i$ is  $r_i$-regular for $1 \le i \le k$.
For $1 \le i \le j \le k$, define $\rho_{i,j} = 1$  if  $\{i,j\}$ is an edge in $\mathcal{H}$ and $0$ otherwise. 
Then $r_i$ is an eigenvalue of $\mathcal{A}(\Gamma_i)$ and
$$
\spec(\mathcal{G}) = \bigg(\bigcup_{i=1}^k \big( \spec(\Gamma_i)\backslash \{r_i\} \big) \bigg)\cup \spec(\widetilde{A}(\mathcal{G})),
$$
where $\widetilde{A}(\mathcal{G})= \begin{bmatrix}
                r_1 & \sqrt{n_1n_2}\rho_{1,2}  & \cdots & \sqrt{n_1n_k}\rho_{1,k}  \\
                \sqrt{n_2n_1}\rho_{2,1}  & r_2 & \cdots & \sqrt{n_2n_k}\rho_{2,k} \\
                \vdots & \vdots & \ddots & \vdots  \\
                \sqrt{n_kn_1}\rho_{k,1} & \sqrt{n_kn_2}\rho_{k,2} & \cdots &r_k
        \end{bmatrix}.$
\end{theorem}
We shall use Theorem \ref{cardaso} to compute the spectrum of various super graphs of dihedral and dicyclic groups. 
We write $\Gamma_A^B$ to denote  the $\partype$ super$\grtype$ graph on $G$ associated to $\partype$, $\grtype$ and $G$. In \cite[Proposition 4.1]{part2}, it was shown that  $\Gamma_A^B$ is isomorphic to  $\Delta[K_{n_1},\ldots,K_{n_k}]$ for some graph $\Delta$ on $k$ vertices, where $n_1,\ldots,n_k$ are the sizes of the equivalence classes	of $\partype$.
Note that the graph $\Delta$ is the induced subgraph of $\Gamma_A^B$ on the set of equivalence class
representatives of $\partype$. Since  $K_{n}$ is $(n -1)$-regular  and 
$$
\spec(K_{n})\backslash \{n - 1\} = \{\underbrace{-1,\dots,-1}_{\text{$(n-1)$~times}}\},
$$ 
applying  Theorem~\ref{cardaso} to $\Gamma_A^B$ we get the spectrum of $\Gamma_A^B$ as given below.
\begin{theorem}\label{t:specmain}
The spectrum   $\Gamma_A^B = \Delta[K_{n_1}, \dots, K_{n_k}]$, where $V(\Delta) = \{1, \dots, k\}$ and $n_i \geq 1$ for $1 \leq i \leq k$, is given by
$$
\spec(\Gamma^B_A) = \bigg(\bigcup_{\substack{1 \le i \le k \\ n_i \geq 1}} \{\underbrace{-1,\dots,-1}_{\text{$(n_i-1)${\rm~times}}}\} \bigg)\cup \spec(\widetilde{\mathcal A}(\Gamma^B_A)),
$$
where $\widetilde{\mathcal A}(\Gamma^B_A)= \begin{bmatrix}
        n_1-1 & \sqrt{n_1n_2}\rho_{1,2}  & \cdots & \sqrt{n_1n_k}\rho_{1,k}  \\
        \sqrt{n_2n_1}\rho_{2,1}  & n_2 - 1 & \cdots & \sqrt{n_2n_k}\rho_{2,k} \\
        \vdots & \vdots & \ddots & \vdots  \\
        \sqrt{n_kn_1}\rho_{k,1} & \sqrt{n_kn_2}\rho_{k,2} & \cdots &n_k - 1
        \end{bmatrix}.$
\end{theorem}

In general it is difficult to compute $\spec(\widetilde{\mathcal A}(G))$. However,  in Section \ref{sspec}, we shall calculate $\spec(\widetilde{\mathcal A}(G))$  explicitly when $\Gamma_A^B$ is equality \emph{super$\grtype$ graph} and conjugacy \emph{super$\grtype$ graph} for dihedral groups and dicyclic group where the graph A is the commuting graph.
In this process,  the concept of \textit{main} function associated with a matrix and a few essential lemmas from \cite{sma21} are useful and these are listed below.
\begin{defn}\rm
Let $M$ be an $n \times n$ complex matrix, and let $u$ and $v$ be $ n \times 1$ complex vectors. The main function associated to the matrix $M$ corresponding to the vectors $u$ and $v$, denoted by $\Gamma_M(u,v)$, is defined to be $\Gamma_M(u,v) = v^{t}(\lambda I - M)^{-1}u \in \mathbb C(\lambda)$. When $u=v$, we denote $\Gamma_M(u,v)$ by $\Gamma_M(u).$ 
\end{defn}

\begin{lemma}\label{schur}
	Let $A, B, C$ and $D$ be matrices such that $M=
	\begin{bmatrix}
		A & B\\
		C & D
	\end{bmatrix}
	$. 
 If $D$ is invertible, then 
	$\det(M) = \det(D) \det(A - BD^{-1}C).$
\end{lemma}
\begin{lemma}\label{mdl}
	Let $A$ be an $n \times n$ invertible matrix, and let $u$ and $v$ be any two $n \times 1$ vectors such that $1 + v^{t}A^{-1}u \ne 0$. Then 
	\begin{enumerate}
		\item $\det (A+uv^t) = (1+ v^t A^{-1} u) \det (A).$
		\item $(A+uv^t)^{-1} = A^{-1} - \dfrac{A^{-1} u v^t A^{-1}}{1+v^tA^{-1}u}.$
	\end{enumerate} 
\end{lemma}

\begin{lemma}\label{evmain}
	Let $M$ be a matrix of order $n$ with an eigenvector $u$ corresponding to the eigenvalue $\mu$. Then $\Gamma_M(u) = \dfrac{\Vert u \Vert^2}{\lambda-\mu}$.
\end{lemma}

We write $\mathrm{ESCom}(G)$ and $\mathrm{CSCom}(G)$ to denote the equality supercommuting and conjugacy supercommuting graphs of a group $G$. In \cite{part2}, $\mathrm{ESCom}(G)$ and $\mathrm{CSCom}(G)$ were realized for dihedral groups and dicyclic groups, as described in the following theorem.
\begin{theorem}\label{ESCom-D2n}
Let   $D_{2n} = \langle a, b : a^n =b^2= e, bab^{-1} = a^{-1}\rangle$ be the dihedral group and $Q_{4n} = \langle a, b : a^{2n} = e, a^n =b^2, bab^{-1} = a^{-1}\rangle$ be the dicyclic group. Then
\[
\mathrm{ESCom}(D_{2n}) \cong \begin{cases}
		K_1 \vee \left(\underbrace{K_1\sqcup \cdots \sqcup K_1}_{\text{$n${\rm~times}}} \sqcup K_{n - 1}\right), &\text{ if  $n$ is odd}\\
		K_2 \vee \left(\underbrace{K_2\sqcup \cdots \sqcup K_2}_{\text{$\frac{n}{2}${\rm~times}}} \sqcup K_{n - 2}\right), &	\text{ if $n$ is even,}
	\end{cases}
\]
 \[
\mathrm{CSCom}(D_{2n}) \cong \begin{cases}
		K_1 \vee \left(K_1 \sqcup K_{\frac{n - 1}{2}}\right)\left[K_1, \underbrace{K_2, \dots, K_2}_{\text{$(\frac{n-1}{2})${\rm~times}}}, K_n\right], \text{ if  $n$ is odd}\\
		K_2 \vee \left(K_1 \sqcup K_1 \sqcup K_{\frac{n}{2} - 1}\right)\left[K_1, K_1, \underbrace{K_2, \dots, K_2}_{(\text{$\frac{n}{2}-1)${\rm~times}}}, K_{\frac{n}{2}}, K_{\frac{n}{2}}\right], 	\text{ if $n$ and $\frac{n}{2}$ are even}\\
		K_2 \vee \left(K_2 \sqcup K_{\frac{n}{2} - 1}\right)\left[K_1, K_1, \underbrace{K_2, \dots, K_2}_{\text{$(\frac{n}{2}-1)${\rm~times}}}, K_{\frac{n}{2}}, K_{\frac{n}{2}}\right], 	\text{ if $n$ is even and $\frac{n}{2}$ is odd,}
	\end{cases}
\]
\[
\mathrm{ESCom}(Q_{4n}) \cong  K_2 \vee \left(\underbrace{K_2\sqcup \cdots \sqcup K_2}_{\text{$n${\rm~times}}} \sqcup K_{2n - 2}\right)
\]
and
\[
\mathrm{C SCom}(Q_{4n}) \cong   \begin{cases}
		K_2 \vee \left(K_1 \sqcup K_1 \sqcup K_{n - 1}\right)\left[K_1, K_1, \underbrace{K_2, \dots, K_2}_{\text{$(n-1)${\rm~times}}}, K_{n}, K_{n}\right], &	\text{ if $n$ is even}	\\
		K_2 \vee \left(K_2 \sqcup K_{n - 1}\right)\left[K_1, K_1, \underbrace{K_2, \dots, K_2}_{\text{$(n-1)${\rm~times}}}, K_{n}, K_{n}\right], &	\text{ if $n$ is odd.}
	\end{cases}
\]
\end{theorem}

\section{Adjacency spectrum of supergraphs}\label{sspec}
In this section, we compute the spectrum of  $\mathrm{ESCom}(G)$ and $\mathrm{CSCom}(G)$, where $G$ is a dihedral or dicyclic group. In our proofs we shall use Theorem \ref{ESCom-D2n}, Theorem \ref{cardaso} along with Schur complement and the main function technique (\cite[Theorem 2]{sma21}). We begin with the computation of $\spec(\mathrm{ESCom}(G))$, where $G$ is the  dihedral group of order $2n$.

\begin{theorem}\label{S-ESCom-D2n}
	Let   $D_{2n} = \langle a, b : a^n =b^2= e, bab^{-1} = a^{-1}\rangle$ be the dihedral group of order $2n$. 
\begin{enumerate}
\item If $n$ is odd then $\spec(\mathrm{ESCom}(D_{2n})) = \{\underbrace{0, \dots,  0}_{\text{$(n-1)${\rm~times}}}, \underbrace{-1, \dots,  -1}_{\text{$(n-2)${\rm~times}}}, \alpha, \beta, \gamma\}$, 
where $\alpha, \beta, \gamma$ are  roots of the equation $x^3 - (n - 2)x^2 -(2n - 1)x + n(n-2) = 0$.
\item If $n$ is even then $\spec(\mathrm{ESCom}(D_{2n})) = \{\underbrace{1, \dots,  1}_{\text{$(\frac{n}{2}-1)${\rm~times}}}, \underbrace{-1, \dots,  -1}_{\text{$(\frac{3n}{2}-2)${\rm~times}}}, \alpha, \beta, \gamma\}$, 
where $\alpha, \beta, \gamma$ are  roots of the equation $x^3 - (n - 1)x^2 -(2n + 1)x + 2n^2 -5n-1 = 0$.
\end{enumerate}	
\end{theorem}
\begin{proof}
(a) If $n$ is odd then by Theorem \ref{ESCom-D2n}, we have
\[
\mathrm{ESCom}(D_{2n}) \cong K_1 \vee (\underbrace{K_1\sqcup \cdots \sqcup K_1}_{n\text{~times}} \sqcup K_{n - 1}).
\]
We identify the vertex set $\{e, a, a^2, \dots, a^{n - 1}, b, ab, \dots, a^{n-1}b\}$ of $\mathrm{ESCom}(D_{2n})$ with the set $\{1, 2, \dots, 2n\}$ preserving the order and observe that

\begin{equation}\label{odd}
	\mathcal{A}(\mathrm{ESCom}(D_{2n})) = \begin{bmatrix}
	0  & 1      & \cdots &1      &1      & 1     &\cdots & 1  \\
	1  & 0      & \cdots &1      &0      & 0     &\cdots & 0 \\
	\vdots & \vdots & \ddots &\vdots &\vdots &\vdots &\ddots & \vdots \\
	1  & 1      & \cdots & 0     & 0     & 0     &\cdots & 0 \\
	1  & 0      & \cdots & 0     & 0     & 0     &\cdots & 0 \\
	\vdots & \vdots & \ddots &\vdots &\vdots &\vdots &\ddots & \vdots \\
	1  & 0      & \cdots & 0     & 0     & 0     &\cdots & 0 \\
\end{bmatrix}.
\end{equation}
We use the Schur complement and the main function technique from \cite[Theorem 2]{sma21} to completely describe the spectrum of the matrix $\mathcal{A}(\mathrm{ESCom}(D_{2n}))$. 

Let $A$ be the adjacency matrix of the complete graph on $n$ vertices. Let $A \ast e_1$ be the $(n+1) \times (n+1)$ matrix obtained from the matrix $A$ as follows.
$$
A \ast e_1 = \begin{bmatrix}
		0  & 1      & \cdots &1   &|   &1       \\
		1  & 0      & \cdots &1    &|  &0      \\
		\vdots & \vdots & \ddots &\vdots &\vdots  \\
		1  & 1      & \cdots & 0   &|  & 0      \\
		-&-&-&-&-&- \\
		1  & 0      & \cdots & 0   &|  & 0      \\
	\end{bmatrix}.
$$
We observe that the matrix $\mathcal{A}(\mathrm{ESCom}(D_{2n}))$ is equal to the matrix 
$$
A_{n*n} := \underbrace{((((A \ast e_1)* e_1)\cdots)*e_1)}_{\text{$n$ times }}.
$$
For example, the matrix $A_{2*4}$ is given below
$$
A_{2*4} = \begin{bmatrix}
	0  & 1  & |    & 1  &1   & 1   &1       \\
	1  & 0   & |  & 0  & 0    & 0  &0      \\
	-  & -   & -  & \cdots & -    & -  &-      \\
	1 & 0 & |    & 0 &0 &0 & 0  \\
	1 & 0 &|    & 0 &0 &0 & 0  \\
	1  & 0 & |    &0  & 0 & 0  & 0    \\
	1  & 0  & |    &0   & 0   &0  & 0      \\
\end{bmatrix}.
$$
Note that $A_{n*0}$ is the adjacency matrix of the graph $K_n$. We observe that
$$
P_{A_{n*n}}(\lambda) :=\det(\lambda I - A_{n*n}) = \det \begin{bmatrix}	\lambda I - A_{n*(n-1)} & - e_{1} \\ -e_{1}^{t} & \lambda
\end{bmatrix}.
$$
Now, by using Lemma \ref{schur} and Lemma \ref{mdl}, we have 
\begin{align*}
	 \det(\lambda I - A_{n*n}) &= \det \begin{bmatrix}	\lambda I - A_{n*(n-1)} & - e_{1} \\ -e_{1}^{t} & \lambda
 \end{bmatrix} \\ &= \lambda \det (\lambda I - A_{n*(n-1)} - \frac{1}{\lambda} e_1 e_1^t) \\ &  = \lambda (1 - \frac{1}{\lambda} T) P_{A_{n*(n-1)}}(\lambda),
\end{align*}
where $T = e_1 ^t (\lambda I - A_{n * (n-1)})^{-1}e_1$. Simplifying this expression for $T$, we get 
$$
T = \frac{\big(\adj(\lambda I - A_{n*(n-1)})\big)_{1,1}}{P_{A_{n*(n-1)}(\lambda) }} = \frac{\lambda^{n-1} P_{A_{(n-1)*0}}(\lambda)}{P_{A_{n*(n-1)}(\lambda)}}.
$$
Substituting the value of $T$ in the previous equation, we get the following recurrence relations.
$$
P_{A_{n*n}}(\lambda)  = \lambda \big( 1 - \frac{\lambda^{n-1} P_{A_{(n-1)*0}}(\lambda)}{\lambda P_{A_{n*(n-1)}}(\lambda)}\big)P_{A_{n * (n-1)}}(\lambda)   = \lambda P_{A_{n * (n-1)}}(\lambda) - \lambda ^{n-1} P_{A_{(n-1)* 0}(\lambda)} 
$$
i.e., 
\begin{equation}
P_{A_{n*n}}(\lambda)  =  \lambda P_{A_{n * (n-1)}}(\lambda) - \lambda ^{n-1} P_{A_{(n-1)* 0}} (\lambda).
\end{equation}
Expanding the middle term recursively, we get 
\begin{equation}\label{relation}
	P_{A_{n*n}}(\lambda)  =  \lambda^n P_{A_{n * 0}}(\lambda) - n \lambda ^{n-1} P_{A_{(n-1)* 0}}(\lambda).
\end{equation}
Equation \eqref{relation} gives  relation between the characteristic polynomial of the adjacency matrix of $\mathrm{ESCom}(D_{2n})$ and the characteristic polynomial of the adjacency matrix of the complete graphs of smaller size. Now, by substituting the values for $P_{A_{n*0}}(\lambda)$ and $P_{A_{(n-1)*0}}$
and simplifying, we get 
\begin{align*}
P_{A_{n*n}}(\lambda)  &=  \lambda^n ((\lambda-(n-1))(\lambda+1)^{n-1}) - n \lambda ^{n-1} ((\lambda-(n-2))(\lambda+1)^{n-2}) \\	
& = \lambda^{n-1} (\lambda + 1)^{n-2} (\lambda(\lambda-(n-1))(\lambda+1) - n (\lambda - (n-2))) \\
& = \lambda^{n-1} (\lambda+1)^{n-2} (\lambda^3 -  (n-2) \lambda^2 - (2n-1) \lambda + n (n-2)).
\end{align*}

\noindent (b) If $n$ is even then by Theorem \ref{ESCom-D2n}, we have
\[
\mathrm{ESCom}(D_{2n})  \cong K_2 \vee (\underbrace{K_2\sqcup \cdots \sqcup K_2}_{\frac{n}{2}\text{~times}} \sqcup K_{n - 2}).
\]
We identify the vertex set $\{e, a^{\frac{n}{2}}, a, a^2, \dots, a^{\frac{n}{2} -1}, a^{\frac{n}{2}+1}, \dots, a^{n-1}, b, a^{\frac{n}{2}}b, ab, a^{\frac{n}{2} + 1}b$, $\dots a^{\frac{n}{2}-1}b, a^{n-1}b\}$ of $\mathrm{ESCom}(D_{2n})$ with the set $\{1, 2, \dots, 2n\}$ preserving the order and observe that 
\begin{equation}\label{even}
	\mathcal{A}(\mathrm{ESCom}(D_{2n})) = \begin{bmatrix}
	0  & 1 &1      & \cdots &1      &1 & 1 & \cdots     & 1    &1  \\
	1  & 0  & 1    & \cdots &1      &1 & 1 &\cdots    & 1    &1    \\
	1  & 1  & 0    & \cdots &1      &0 & 0 &\cdots     & 0    &0    \\
	\vdots & \vdots & \vdots & \ddots &\vdots &\vdots & \vdots & \ddots &\vdots &\vdots   \\
	1  & 1   & 1   & \cdots & 0     & 0  &0 & \cdots   & 0     &0  \\
	1  & 1  & 0   & \cdots & 0     & 0  &1 & \cdots  & 0    &0 \\
	1  & 1   & 0   & \cdots & 0     & 1 & 0 & \cdots    & 0    &0 \\
	\vdots & \vdots & \vdots & \ddots &\vdots &\vdots & \vdots &\ddots &\vdots & \vdots   \\
	1  & 1      & 0 & \cdots     & 0     & 0   &0 & \cdots &0 & 1 \\
	1  & 1   & 0   & \cdots & 0     & 0  &0 & \cdots  & 1    &0 \\
\end{bmatrix}.
\end{equation}

Again, we use the Schur complement and the main function technique as in the previous case. Since the calculation is the same, we skip it. This will give us the following characteristic polynomial of $\mathcal{A}(\mathrm{ESCom}(D_{2n}))$ 
\[
(x-1)^{\frac{n}{2} - 1}(x+1)^{\frac{3n}{2} - 2}(x^3 - (n - 1)x^2 -(2n + 1)x + 2n^2 -5n-1).
\]
This completes the proof. 
\end{proof}
\begin{rem}\rm
In the proof of the above theorem, the characteristic polynomial of the matrix given in Equation \ref{odd} can be calculated directly by expanding the determinant of $(\lambda I - A)$ along the last column. But the same method gets complicated in the case of the matrix given in Equation \ref{even}. So we use Schur complement and the main function technique, which works for odd and even cases uniformly, and the resulting calcualtions are much simpler.
\end{rem}
\begin{theorem}\label{S-ESCom-Q4n}
Let  $Q_{4n} = \langle a, b : a^{2n} = e, a^n =b^2, bab^{-1} = a^{-1}\rangle$ the dicyclic group. Then
\[
\spec(\mathrm{ESCom}(Q_{4n})) = \{\underbrace{1, \dots,  1}_{\text{$(n-1)${\rm~times}}}, \underbrace{-1, \dots,  -1}_{\text{$(3n-2)${\rm~times}}}, \alpha, \beta, \gamma\}, 
\]
where $\alpha, \beta, \gamma$ are  roots of the equation $x^3 - (2n - 1)x^2 -(4n + 1)x + 8n^2 -10n-1= 0$.
\end{theorem}

\begin{proof}
By Theorem \ref{ESCom-D2n} we have 
$\mathrm{ESCom}(Q_{4n}) \cong \mathrm{ESCom}(D_{2\times 2n})$. Hence, the result follows from Theorem \ref{S-ESCom-D2n}(b).
\end{proof}

\begin{theorem}\label{S-CSCom-D2n}
	Let   $D_{2n} = \langle a, b : a^n =b^2= e, bab^{-1} = a^{-1}\rangle$ be the dihedral group of order $2n$. 
\begin{enumerate}
\item If   $n$ is odd then $\spec(\mathrm{CSCom}(D_{2n})) = \{\underbrace{-1,\dots,-1}_{\text{$(2n-3)${\rm~times}}}, \alpha, \beta, \gamma \,\}$, where $\alpha, \beta, \gamma$ are the roots of the equation $x^3+(3-2n)x^2+(n^2-5n+3)x+ 2n^2-4n+1$.

\item If   $n$ and $\frac{n}{2}$  are even then $\spec(\mathrm{CSCom}(D_{2n})) = \{ \, \underbrace{-1,\dots,-1}_{\text{$(2n - 4)${\rm~times}}}, \, \frac{n}{2} -1, \alpha, \beta, \gamma \,\}$, where $\alpha, \beta, \gamma$ are the roots of the equation $x^3+(3-\frac{3n}{2})x^2+(\frac{n^2}{2}-5n+3)x+\frac{5n^2}{2}-\frac{15n}{2}+1 = 0$.

\item If   $n$ is even and $\frac{n}{2}$ is odd then $\spec(\mathrm{CSCom}(D_{2n})) = \{ \, \underbrace{-1,\dots,-1}_{\text{$(2n - 4)${\rm~times}}} \, , \alpha, \beta, \gamma, \delta \,\}$, where $\alpha, \beta, \gamma, \delta$ are the roots of the equation $x^4 + (4-2n)x^3 + (n^2-8n+6)x^2 + (4n^2-14n+4)x +3n^2-8n+1 = 0$.
	
\end{enumerate}
\end{theorem}
\begin{proof}
(a) If $n$ is odd then, by Theorem 	\ref{ESCom-D2n}, we have 
$$
\mathrm{CSCom}(D_{2n}) \cong \Delta[K_1, \underbrace{K_2, \dots, K_2}_{(\frac{n-1}{2})\text{~times}}, K_n],
$$ 
where $\Delta = K_1 \vee (K_1 \sqcup K_{\frac{n - 1}{2}})$ = $K_1 \vee ( K_{\frac{n - 1}{2}}\sqcup K_1 )$. Therefore,
$$\mathcal{A}(\Delta) = \begin{bmatrix}
	0 & 1  & \cdots &1 & 1  \\
	1  & 0 & \cdots &1 & 0 \\
	\vdots & \vdots & \ddots & \vdots & \vdots  \\
	1 & 1  & \cdots  & 0 & 0  \\
	1 & 0 & \cdots & 0 &0
\end{bmatrix}
\text{ and so }
\rho_{i,j} = \begin{cases}
	0, & \text{ if } i = j \text{ or } \\
	& 2 \leq i \leq \frac{n +1 }{2} \text{ and } j =  \frac{n + 3}{2} \text{ or }\\
	& i = \frac{n + 3}{2}   \text{ and } 2 \leq j \leq \frac{n +1 }{2}\\
	1,  & \text{ otherwise. }
\end{cases}$$
Hence, by Theorem  \ref{t:specmain}, it follows that 
$$\spec(\mathrm{CSCom}(D_{2n})) =  \{\underbrace{-1,\dots,-1}_{(\frac{3n -3}{2})\text{~times}}\} \cup \spec(\widetilde{\mathcal A}(\mathrm{CSCom}(D_{2n}))),$$ 
where $\widetilde{\mathcal A}(\mathrm{CSCom}(D_{2n}))$  is a matrix of size $\frac{n + 3}{2}$ given by 
\begin{align*}
\widetilde{\mathcal A}(\mathrm{CSCom}(D_{2n})) &= \begin{bmatrix}
	0 & \sqrt{2} & \sqrt{2} & \cdots &\sqrt{2} & \sqrt{n}  \\
	\sqrt{2}  & 1 & 2 & \cdots &2 & 0 \\
 \sqrt{2} & 2 & 1 & \cdots & 2 & 0\\
	\vdots & \vdots & \vdots & \ddots & \vdots & \vdots \\
	\sqrt{2} & 2 &2  & \cdots  & 1 & 0  \\
	\sqrt{n} & 0 & 0 & \cdots & 0 &n - 1
\end{bmatrix}\\
&= \begin{bmatrix}
	0 & | & \sqrt{2} & \sqrt 2  & \cdots &\sqrt{2} &\sqrt{2} & | & \sqrt{n}  \\
		- & -  & - &- & - & - & - & - & - \\
	\sqrt{2} & |   & 1 &  2 & \cdots &2 &2 & | & 0 \\
		\sqrt{2} & |   &2 & 1 & \cdots &2 &2 & | & 0 \\
	\vdots & |  & \vdots & \vdots & \ddots & \vdots & \vdots & |  & 0\\
		\sqrt{2} & |   &2 & \vdots & \cdots &1 &2 & | & \vdots \\
	\sqrt{2} & |  &2  & \cdots & \cdots & 2&1 & | & 0  \\
		- & -  & - & - &- & - & - & - &-\\
	\sqrt{n} & |  & 0 & \cdots & 0 & 0 & 0 & | &n - 1
\end{bmatrix}.
\end{align*}
Note that the middle block of $\widetilde{\mathcal A}(\mathrm{CSCom}(D_{2n}))$ is $2 J_{\frac{n-1}{2}} - I_{\frac{n-1}{2}}$ whose eigenvalues are  $-1$ with multiplicity $\frac{n-3}{2}=\frac{n-1}{2}-1$ and $n-2$ with multiplicity $1$. These eigenvalues are useful in obtaining the eigenvalues of $\widetilde{\mathcal A}(\mathrm{CSCom}(D_{2n}))$. We obtain the characteristic polynomial of $\widetilde{\mathcal A}(\mathrm{CSCom}(D_{2n}))$ by expanding along the last column as given below:

\begin{align*}
&P_{\widetilde{\mathcal A}(\mathrm{CSCom}(D_{2n}))}(\lambda)\\
&~~= (-1)^{2+\frac{n+3}{2}} \sqrt{n} \det \begin{bmatrix}
	 - \sqrt{2} & |   & \lambda-1 &-2  &  \cdots &  \cdots &-2 \\
	-\sqrt 2 & |  & - 2 &\lambda-1 & \cdots & \cdots & \vdots \\
 -\sqrt{2} & | & -2 &-2 & \cdots & \cdots & \vdots\\
		\vdots & |  & \vdots & \vdots & \ddots & \vdots & \vdots \\
	- \sqrt 2 & |  & -2  & \cdots  & \cdots & - 2 & \lambda-1  \\
	- & -  & - &- & - & - & - \\
	-\sqrt{n} & |  & 0 & \cdots & \cdots &0&0
\end{bmatrix} \\
& \qquad \qquad \qquad \qquad + (-1)^{n+3}(\lambda-(n-1)) \det \begin{bmatrix}
\lambda & | & -\sqrt{2}  & -\sqrt{2}  & \cdots &-\sqrt{2}  \\
- & -  & - &- & -& -   \\
-\sqrt{2} & |   & \lambda - 1 &-2& \cdots &-2  \\
\vdots & |  & \vdots & \vdots & \vdots & \vdots  \\
-\sqrt{2} & |  & -2& \cdots  & -2& \lambda - 1    \\  
\end{bmatrix} \\ 
& ~~= (-1)^{(2+\frac{n+3}{2}+2+\frac{n+1}{2})} n (\lambda+1)^{\frac{n-3}{2}}\cdot(\lambda-(n-2)) + (-1)^{n+3} (\lambda-(n-1)) \det (\lambda I - \bar{A})\\
& ~~=(-1)\cdot n \cdot(\lambda+1)^{\frac{n-3}{2}}\cdot(\lambda-(n-2)) +  (\lambda-(n-1)) \det (\lambda I - \bar{A}),
\end{align*}
where
 $$
\bar{A}:= \begin{bmatrix}
0 & | & \sqrt{2}  & \sqrt{2}  & \cdots &\sqrt{2}  \\
- & -  & - &- & -& -   \\
\sqrt{2} & |   & 1 &2& \cdots &2  \\
\vdots & |  & \vdots & \vdots & \vdots & \vdots  \\
\sqrt{2} & |  & 2& \cdots  & 2& 1    \\  
\end{bmatrix}.
$$

To calculate the characteristic polynomial of $\bar{A}$, we consider $\bar{A}$ as follows.
$$ \bar{A} = \begin{bmatrix}
0 & | & \sqrt{2}  & \sqrt{2}  & \cdots &\sqrt{2}  \\
- & -  & - &- & -& -   \\
\sqrt{2} & |   & 1 &2& \cdots &2  \\
\vdots & |  & \vdots & \vdots & \vdots & \vdots  \\
\sqrt{2} & |  & 2& \cdots  & 2& 1    \\  
\end{bmatrix} := \begin{bmatrix}
A & | & B \\ -- & | & -- \\ C & | & D
\end{bmatrix} (say).$$
By Lemma \ref{schur}, we have
$$\det (\lambda I - \bar{A}) = \det (\lambda I - ( 2 J_{\frac{n-1}{2}} - I_{\frac{n-1}{2}})) (\lambda - \Gamma_D(\sqrt 2 \cdot \mathbf{1}_{\frac{n-1}{2}})).$$
Since $\sqrt 2 \cdot \mathbf{1}_{\frac{n-1}{2}}$ is an eigenvector of the matrix $ 2 J_{\frac{n-1}{2}} - I_{\frac{n-1}{2}}$ corresponding to the eigenvalue $n-2$, by Lemma \ref{evmain}, we have $$\Gamma_D(\sqrt 2 \cdot \mathbf{1}_{\frac{n-1}{2}}) = \frac{||\sqrt 2 \cdot \mathbf{1}_{\frac{n-1}{2}}||^2}{\lambda -(n-2))} = \frac{n-1}{\lambda - (n-2)} .$$
Substituting this value of the main function in the above equation, we get 

\begin{align*}
	\det (\lambda I - \bar{A}) &= \det (\lambda I - ( 2 J_{\frac{n-1}{2}} - I_{\frac{n-1}{2}})) \Big(\lambda -  \frac{n-1}{\lambda - (n-2)}\Big) \\
	&= \det (\lambda I - ( 2 J_{\frac{n-1}{2}} - I_{\frac{n-1}{2}})) \Big(\frac{\lambda^2-\lambda(n-2)-(n-1)}{\lambda - (n-2)}\Big) \\
	&= (\lambda+1)^{(\frac{n-3}{2})} (\lambda-(n-2))\Big(\frac{\lambda^2-\lambda(n-2)-(n-1)}{\lambda - (n-2)}\Big)\\
 &=(\lambda+1)^{(\frac{n-3}{2})}(\lambda^2-\lambda(n-2)-(n-1)).
\end{align*}
which is the required characteristic polynomial of the matrix $\bar{A}$.
Therefore, \begin{align*}
    &P_{\widetilde{\mathcal A}(\mathrm{CSCom}(D_{2n}))}(\lambda)\\
    &~~=(-1)\cdot n \cdot(\lambda+1)^{\frac{n-3}{2}}\cdot(\lambda-(n-2)) +  (\lambda-(n-1)) (\lambda+1)^{(\frac{n-3}{2})}(\lambda^2-\lambda(n-2)-(n-1))\\
    &~~=(\lambda+1)^{\frac{n-3}{2}}\cdot\big(\lambda^3+(3-2n)\lambda^2+(n^2-5n+3)\lambda+(2n^2-4n+1)\big).
\end{align*}

\noindent (b)  If   $n$ and $\frac{n}{2}$  are even then, by Theorem 	\ref{ESCom-D2n}, we have 
$$\mathrm{CSCom}(D_{2n}) \cong \Delta[K_1, K_1, \underbrace{K_2, \dots, K_2}_{(\frac{n}{2}-1)\text{~times}}, K_{\frac{n}{2}}, K_{\frac{n}{2}}],$$
where $\Delta = K_2 \vee (K_1 \sqcup K_1 \sqcup K_{\frac{n}{2} - 1})$. Therefore,
$$\mathcal{A}(\Delta) = \begin{bmatrix}
	0 & 1 & 1      & \cdots & 1       & 1      & 1\\
	1 & 0 & 1      & \cdots & 1       & 1      & 1\\
	1 & 1 & 0      & \cdots & 1       & 0      & 0\\
	\vdots & \vdots     & \vdots & \ddots & \vdots & \vdots & \vdots\\
	1 & 1 &1       & \cdots & 0      & 0       & 0\\
	1 & 1 &0      & \cdots & 0       & 0       & 0\\
	1 & 1 &0      & \cdots & 0       & 0       & 0
\end{bmatrix}
\text{ and so }
\rho_{i,j} = \begin{cases}
	0, &\!\!\!\!\!\! \text{ if } i = j \text{ or } \\
	&\!\!\!\!\! 3 \leq i \leq \frac{n}{2} + 2 \text{ and } j =  \frac{n}{2} + 2,  \frac{n}{2} + 3  \text{ or }\\
	&\!\!\!\!\! i = \frac{n}{2} + 2,  \frac{n}{2} + 3   \text{ and } 3 \leq j \leq \frac{n}{2}+2\\
	1,  &\!\!\!\!\! \text{ otherwise. }
\end{cases}$$
Hence, by Theorem  \ref{t:specmain}, it follows that 
$$\spec(\mathrm{CSCom}(D_{2n})) =  \{\underbrace{-1,\dots,-1}_{\text{($\frac{3n}{2} - 3)$~times}}\} \cup \spec(\widetilde{\mathcal A}(\mathrm{CSCom}(D_{2n}))),$$ 
where $\widetilde{\mathcal A}(\mathrm{CSCom}(D_{2n}))$  is a matrix of size $\frac{n}{2} + 3$ given by 
$$\widetilde{\mathcal A}(\mathrm{CSCom}(D_{2n}))= \begin{bmatrix}
	0 &1        & \sqrt{2} &\sqrt 2 & \cdots &\sqrt{2} & \sqrt{\frac{n}{2}}  &\sqrt{\frac{n}{2}}\\
	1 &0        & \sqrt{2}  &\sqrt 2& \cdots &\sqrt{2} & \sqrt{\frac{n}{2}}  &\sqrt{\frac{n}{2}}\\
	\sqrt{2} &\sqrt{2} & 1       &2  & \cdots &2 & 0 & 0\\
 \sqrt 2& \sqrt 2& 2 & 1 & \cdots & 2 & 0 & 0\\
	\vdots & \vdots  & \vdots & \vdots & \ddots & \vdots  &  \vdots&\vdots\\
	\sqrt{2} &\sqrt{2}& 2 & \cdots  &2& 1 & 0  & 0\\
	\sqrt{\frac{n}{2}} &\sqrt{\frac{n}{2}}& 0 & \cdots &\cdots& 0 & \frac{n}{2} - 1 & 0\\
	\sqrt{\frac{n}{2}} &\sqrt{\frac{n}{2}}& 0 & \cdots &\cdots& 0 &0 &\frac{n}{2} - 1
\end{bmatrix}.$$

By taking the bottom right corner submatrix of size $2 \times 2$ as $D$, and applying Lemma \ref{schur}, putting $A=\lambda I-A'$, where $A'$ is the top left
submatrix of size $\left(\frac{n}{2}+1\right)\times\left(\frac{n}{2}+1\right)$ submatrix of $\widetilde{\mathcal{A}}(\mathrm{CSCom}(D_{2n})$ given earlier, we get 
\begin{align*}
P_{\widetilde{\mathcal A}(\mathrm{CSCom}(D_{2n}))}(\lambda)&=
\det\begin{bmatrix}
    \lambda-(\frac{n}{2}-1) & 0 \\
    0 & \lambda-(\frac{n}{2}-1)
\end{bmatrix} \times\\
&\det \Big( A - \begin{bmatrix}
    \sqrt{\frac{n}{2}} & \sqrt{\frac{n}{2}} \\
    \sqrt{\frac{n}{2}} & \sqrt{\frac{n}{2}} \\
    0 & 0 \\
    \vdots & \vdots\\
    0 & 0
\end{bmatrix}\begin{bmatrix}
    \frac{1}{\lambda-(\frac{n}{2}-1)} & 0 \\
    0 & \frac{1}{\lambda-(\frac{n}{2}-1)}
\end{bmatrix}\begin{bmatrix}
      \sqrt{\frac{n}{2}} &   \sqrt{\frac{n}{2}} & 0 & \cdots & 0 \\
        \sqrt{\frac{n}{2}} &   \sqrt{\frac{n}{2}} & 0 & \cdots & 0 
\end{bmatrix}\Big).
 \end{align*}
Letting  $a =  \frac{1}{\lambda-(\frac{n}{2}-1)}$, we get
\begin{align*}
P_{\widetilde{\mathcal A}(\mathrm{CSCom}(D_{2n}))}(\lambda)&=\frac{1}{a^2} \det \Big( A - \begin{bmatrix}
    \sqrt{\frac{n}{2}}a & \sqrt{\frac{n}{2}}a \\
    \sqrt{\frac{n}{2}}a & \sqrt{\frac{n}{2}}a \\
    0 & 0 \\
    \vdots & \vdots\\
    0 & 0
\end{bmatrix}\begin{bmatrix}
      \sqrt{\frac{n}{2}} &   \sqrt{\frac{n}{2}} & 0 & \cdots & 0 \\
        \sqrt{\frac{n}{2}} &   \sqrt{\frac{n}{2}} & 0 & \cdots & 0 
\end{bmatrix}\Big)\\
&= \frac{1}{a^2} \det \Big(A-\begin{bmatrix}
    na & na & 0 & \cdots & 0 \\
    na & na & 0 & \cdots & 0 \\
    0 & 0 & \vdots & \vdots & \vdots \\
    \vdots & \vdots & \vdots & \vdots & \vdots \\
    0 & 0 & \cdots & \cdots & 0
\end{bmatrix}\Big)\\
&=\frac{1}{a^2} \det \Big(\begin{bmatrix}
\lambda-na & -1-na& -\sqrt{2} & -\sqrt{2} & \cdots &-\sqrt{2}\\
-1-na & \lambda-na&-\sqrt{2} & -\sqrt{2} & \cdots &-\sqrt{2}\\
	-\sqrt{2}  & -\sqrt 2&  \lambda-1 & -2 & \cdots &-2 \\
 -\sqrt{2} & -\sqrt 2& -2 &  \lambda-1 & \cdots & -2 \\
	\vdots &\vdots& \vdots & \vdots & \ddots & \vdots \\
	-\sqrt{2} &-\sqrt 2& -2 &-2  & \cdots  &  \lambda-1 
\end{bmatrix}.
 \end{align*}
Again using Lemma \ref{schur}, this time taking the bottom right corner submatrix of size $\left(\frac{n}{2}-1\right)\times\left(\frac{n}{2}-1\right)$ as $D$, we get
\begin{align*}
&P_{\widetilde{\mathcal A}(\mathrm{CSCom}(D_{2n}))}(\lambda)=\frac{1}{a^2}(\lambda+1)^{\frac{n}{2}-2} (\lambda-(n-3))\times\\
&~~~~~~~~~~~~~\det \Big(\begin{bmatrix}
        \lambda-na & -1-na \\
        -1-na & \lambda-na
    \end{bmatrix}-\begin{bmatrix}\Gamma_{(\lambda I_{\frac{n}{2}-1}-D)}(\sqrt{2}\cdot \mathbf{1}_{\frac{n}{2}-1}) &\Gamma_{(\lambda I_{\frac{n}{2}-1}-D)}(\sqrt{2}\cdot \mathbf{1}_{\frac{n}{2}-1})  \\
  \Gamma_{(\lambda I_{\frac{n}{2}-1}-D)}(\sqrt{2}\cdot \mathbf{1}_{\frac{n}{2}-1})  & \Gamma_{(\lambda I_{\frac{n}{2}-1}-D)}(\sqrt{2}\cdot \mathbf{1}_{\frac{n}{2}-1}) 
    \end{bmatrix}\Big).
\end{align*}
    Since $ \sqrt{2}\cdot \mathbf{1}_{\frac{n}{2}-1}$ is an eigenvector of ${(\lambda I_{\frac{n}{2}-1}-D)}$ corresponding to the eigenvalue $n-3$, by Lemma \ref{evmain},  we have 
\begin{align*} 
& P_{\widetilde{\mathcal A}(\mathrm{CSCom}(D_{2n}))}(\lambda) \\
 &\indent =\frac{1}{a^2}(\lambda+1)^{\frac{n}{2}-2} (\lambda-(n-3)) \det \Big(\begin{bmatrix}
        \lambda-na & -1-na \\
        -1-na & \lambda-na
    \end{bmatrix}-\begin{bmatrix}\frac{n-2}{\lambda-(n-3)}&\frac{n-2}{\lambda-(n-3)}\\\frac{n-2}{\lambda-(n-3)}&\frac{n-2}{\lambda-(n-3)}
    \end{bmatrix}\Big)\\
&\indent=\frac{1}{a^2}(\lambda+1)^{\frac{n}{2}-2} (\lambda-(n-3))\Big[\Big((\lambda-na)-\frac{(n-2)}{\lambda-(n-3)}\Big)^2- \Big((-1-na)-\frac{(n-2)}{\lambda-(n-3)}\Big)^2\Big]\\
&\indent=\frac{1}{a^2}(\lambda+1)^{\frac{n}{2}-2} (\lambda-(n-3)) \Big(\lambda-1-2na-\frac{2(n-2)}{\lambda-(n-3)}\Big)(\lambda+1)\\
&\indent=\frac{1}{a^2}(\lambda+1)^{\frac{n}{2}-1} \Big((\lambda-1-2na)(\lambda-(n-3))-2(n-2)\Big)\\
&\indent=(\lambda-(\frac{n}{2}-1))(\lambda+1)^{\frac{n}{2}-1}\Big( (\lambda^3+(3-\frac{3n}{2})\lambda^2+(\frac{n^2}{2}-5n+3)\lambda+(\frac{5n^2}{2}-\frac{15n}{2}+1)\Big).
\end{align*}
Hence, the result follows.

\noindent (c)  If   $n$ even and $\frac{n}{2}$  is odd then, by Theorem 	\ref{ESCom-D2n}, we have 
$$\mathrm{CSCom}(D_{2n}) \cong \Delta[K_1, K_1, \underbrace{K_2, \dots, K_2}_{(\frac{n}{2}-1)\text{~times}}, K_{\frac{n}{2}}, K_{\frac{n}{2}}],$$
where $\Delta = K_2 \vee (K_2  \sqcup K_{\frac{n}{2} - 1})$. Therefore,
 $$\mathcal{A}(\Delta) = \begin{bmatrix}
 	0 & 1 & 1      & \cdots & 1       & 1      & 1\\
 	1 & 0 & 1      & \cdots & 1       & 1      & 1\\
 	1 & 1 & 0      & \cdots & 1       & 0      & 0\\
 	\vdots & \vdots     & \vdots & \ddots & \vdots & \vdots & \vdots\\
 	1 & 1 &1       & \cdots & 0      & 0       & 0\\
 	1 & 1 &0      & \cdots & 0       & 0       & 1\\
 	1 & 1 &0      & \cdots & 0       & 1       & 0
 \end{bmatrix}
 \text{ and so }
 \rho_{i,j} = \begin{cases}
 	0, &\!\!\!\!\!\! \text{ if } i = j \text{ or } \\
 	&\!\!\!\!\! 3 \leq i \leq \frac{n}{2} + 1 \text{ and } j =  \frac{n}{2} + 2,  \frac{n}{2} + 3  \text{ or }\\
 	&\!\!\!\!\! i = \frac{n}{2} + 2,  \frac{n}{2} + 3   \text{ and } 3 \leq j \leq \frac{n}{2}+1\\
 	1,  &\!\!\!\!\! \text{ otherwise. }
 \end{cases}$$
 Hence, by Theorem  \ref{t:specmain}, it follows that
 $$\spec(\mathrm{CSCom}(D_{2n})) =  \{\underbrace{-1,\dots,-1}_{\text{($ \frac{3n}{2} - 3)$~times}}\} \cup \spec(\widetilde{\mathcal A}(\mathrm{CSCom}(D_{2n}))),$$ 
 where $\widetilde{\mathcal A}(\mathrm{CSCom}(D_{2n}))$  is a matrix of size $\frac{n}{2} + 3$ given by 
  $$\widetilde{\mathcal A}(\mathrm{CSCom}(D_{2n}))= \begin{bmatrix}
	0 &1        & \sqrt{2} &\sqrt 2 & \cdots &\sqrt{2} & \sqrt{\frac{n}{2}}  &\sqrt{\frac{n}{2}}\\
	1 &0        & \sqrt{2}  &\sqrt 2& \cdots &\sqrt{2} & \sqrt{\frac{n}{2}}  &\sqrt{\frac{n}{2}}\\
	\sqrt{2} &\sqrt{2} & 1       &2  & \cdots &2 & 0 & 0\\
 \sqrt 2& \sqrt 2& 2 & 1 & \cdots & 2 & 0 & 0\\
	\vdots & \vdots  & \vdots & \vdots & \ddots & \vdots  &  \vdots&\vdots\\
	\sqrt{2} &\sqrt{2}& 2 & \cdots  &2& 1 & 0  & 0\\
	\sqrt{\frac{n}{2}} &\sqrt{\frac{n}{2}}& 0 & \cdots &\cdots& 0 & \frac{n}{2} - 1 & \frac{n}{2}\\
	\sqrt{\frac{n}{2}} &\sqrt{\frac{n}{2}}& 0 & \cdots &\cdots& 0 &\frac{n}{2} &\frac{n}{2} - 1
\end{bmatrix}.$$

As before, by taking $b=\frac{n}{2}-1$ and the bottom right corner submatrix of size $2 \times 2$ as $D$, and applying Lemma \ref{schur}, we get 
\begin{align*}
&P_{\widetilde{\mathcal A}(\mathrm{CSCom}(D_{2n}))}(\lambda)\\
&\indent = \det\begin{bmatrix}
    \lambda-b& -\frac{n}{2} \\
    -\frac{n}{2} & \lambda-b
\end{bmatrix}\det \Big( A - \begin{bmatrix}
    \sqrt{\frac{n}{2}} & \sqrt{\frac{n}{2}} \\
    \sqrt{\frac{n}{2}} & \sqrt{\frac{n}{2}} \\
    0 & 0 \\
    \vdots & \vdots\\
    0 & 0
\end{bmatrix}\begin{bmatrix}
    \lambda-b& -\frac{n}{2} \\
    -\frac{n}{2} & \lambda-b
\end{bmatrix}^{-1}\begin{bmatrix}
      \sqrt{\frac{n}{2}} &   \sqrt{\frac{n}{2}} & 0 & \cdots & 0 \\
        \sqrt{\frac{n}{2}} &   \sqrt{\frac{n}{2}} & 0 & \cdots & 0 
\end{bmatrix}\Big),
 \end{align*}

\begin{align*}
&=\Big((\lambda-b)^2-{\frac{n}{4}}^2\Big)\det \Big( A - \begin{bmatrix}
    \sqrt{\frac{n}{2}} & \sqrt{\frac{n}{2}} \\
    \sqrt{\frac{n}{2}} & \sqrt{\frac{n}{2}} \\
    0 & 0 \\
    \vdots & \vdots\\
    0 & 0
\end{bmatrix}\frac{1}{\Big((\lambda-b)^2-\frac{n^2          }{4}\Big)}\begin{bmatrix}
  \lambda-b & \frac{n}{2}\\
  \frac{n}{2} & \lambda-b
  \end{bmatrix}\begin{bmatrix}
      \sqrt{\frac{n}{2}} &   \sqrt{\frac{n}{2}} & 0 & \cdots & 0 \\
        \sqrt{\frac{n}{2}} &   \sqrt{\frac{n}{2}} & 0 & \cdots & 0 
\end{bmatrix}\Big)\\
&=\Big((\lambda-b)^2-\frac{n  ^2       }{4}\Big)\det \Big( A - \frac{1}{\Big((\lambda-b)^2-\frac{n  ^2       }{4}\Big)}\begin{bmatrix}
    \sqrt{\frac{n}{2}} & \sqrt{\frac{n}{2}} \\
    \sqrt{\frac{n}{2}} & \sqrt{\frac{n}{2}} \\
    0 & 0 \\
    \vdots & \vdots\\
    0 & 0
\end{bmatrix}\begin{bmatrix}
  \lambda-b & \frac{n}{2}\\
\frac{n}{2} & \lambda-b
  \end{bmatrix}\begin{bmatrix}
      \sqrt{\frac{n}{2}} &   \sqrt{\frac{n}{2}} & 0 & \cdots & 0 \\
        \sqrt{\frac{n}{2}} &   \sqrt{\frac{n}{2}} & 0 & \cdots & 0 
\end{bmatrix}\Big)
\end{align*}
\begin{align*}
&=\Big((\lambda-b)^2-\frac{n ^2        }{4}\Big)\times\\
&~~~~~~~~~~~~~~\det \Big( A - \frac{1}{\Big((\lambda-b)^2-\frac{n  ^2       }{4}\Big)}\begin{bmatrix}
    2\Big((\lambda-b)\frac{n}{2}+(\frac{n^2}{4})\Big) &   2\Big((\lambda-b)\frac{n}{2}+(\frac{n^2}{4})\Big)  & 0 & \cdots & 0 \\
     2\Big((\lambda-b)\frac{n}{2}+(\frac{n^2}{4})\Big)  &   2\Big((\lambda-b)\frac{n}{2}+(\frac{n^2}{4
})\Big)       & 0 & \cdots & 0 \\
    0 & 0 & \vdots & \vdots & \vdots \\
    \vdots & \vdots & \vdots & \vdots & \vdots \\
    0 & 0 & \cdots & \cdots & 0
\end{bmatrix}\Big)\\
&=\Big((\lambda-b)^2-\frac{n       ^2  }{4}\Big)\det \Big(\begin{bmatrix}
\lambda-\frac{  2\Big((\lambda-b)\frac{n}{2}+(\frac{n^2}{4})\Big)}{\Big((\lambda-b)^2-\frac{n ^2        }{4}\Big)} & -1-\frac{  2\Big((\lambda-b)\frac{n}{2}+(\frac{n^2}{4})\Big)}{\Big((\lambda-b)^2-\frac{n    ^2     }{4}\Big)}& -\sqrt{2} & -\sqrt{2} & \cdots &-\sqrt{2}\\
-1- \frac{  2\Big((\lambda-b)\frac{n}{2}+(\frac{n^2}{4})\Big)}{\Big((\lambda-b)^2-\frac{n   ^2      }{4}\Big)}& \lambda-\frac{  2\Big((\lambda-b)\frac{n}{2}+(\frac{n^2}{4})\Big)}{\Big((\lambda-b)^2-\frac{n    ^2      }{4}\Big)}&-\sqrt{2} & -\sqrt{2} & \cdots &-\sqrt{2}\\
	-\sqrt{2}  & -\sqrt 2&  \lambda-1 & -2 & \cdots &-2 \\
 -\sqrt{2} & -\sqrt 2& -2 &  \lambda-1 & \cdots & -2 \\
	\vdots &\vdots& \vdots & \vdots & \ddots & \vdots \\
	-\sqrt{2} &-\sqrt 2& -2 &-2  & \cdots  &  \lambda-1 
\end{bmatrix}\Big).
 \end{align*}
 
\noindent Again using Lemma \ref{schur}, this time taking the bottom right corner submatrix of size $\frac{n}{2}-1\times \frac{n}{2}-1$ as $D$, we get

\begin{align*}
&P_{\widetilde{\mathcal A}(\mathrm{CSCom}(D_{2n}))}(\lambda)\\
&=\Big((\lambda-b)^2-\frac{n  ^2       }{4}\Big)(\lambda+1)^{\frac{n}{2}-2} (\lambda-(n-3)) \det \Big(\begin{bmatrix}
        \lambda-\frac{  2\Big((\lambda-b)\frac{n}{2}+(\frac{n^2}{4})\Big)}{\Big((\lambda-b)^2-\frac{n   ^2      }{4}\Big)}& -1-\frac{  2\Big((\lambda-b)\frac{n}{2}+(\frac{n^2}{4})\Big)}{\Big((\lambda-b)^2-\frac{n ^2        }{4}\Big)} \\
        -1-\frac{  2\Big((\lambda-b)\frac{n}{2}+(\frac{n^2 }{4})\Big)}{\Big((\lambda-b)^2-\frac{n         }{2}\Big)} & \lambda-\frac{  2\Big((\lambda-b)\frac{n}{2}+(\frac{n^2}{4})\Big)}{\Big((\lambda-b)^2-\frac{n     ^2    }{4}\Big)}
    \end{bmatrix}\\
    &~~~~~~~~~~~-\begin{bmatrix}\Gamma_{\lambda I_{\frac{n}{2}-1} -D}(\sqrt{2}\cdot \mathbf{1}_{\frac{n}{2}-1}) &\Gamma_{\lambda I_{\frac{n}{2}-1} -D}(\sqrt{2}\cdot \mathbf{1}_{\frac{n}{2}-1})  \\
  \Gamma_{\lambda I_{\frac{n}{2}-1} -D}(\sqrt{2}\cdot \mathbf{1}_{\frac{n}{2}-1})  & \Gamma_{\lambda I_{\frac{n}{2}-1} -D}(\sqrt{2}\cdot \mathbf{1}_{\frac{n}{2}-1}) 
    \end{bmatrix}\Big).
\end{align*}
    Since $ \sqrt{2}\cdot \mathbf{1}_{\frac{n}{2}-1}$ is an eigenvector of ${\lambda I_{\frac{n}{2}-1} -D}$ corresponding to the eigenvalue $n-3$, By Lemma \ref{evmain},  we have 
\begin{align*}    
&P_{\widetilde{\mathcal A}(\mathrm{CSCom}(D_{2n}))}(\lambda)=\Big((\lambda-b)^2-\frac{n   ^2  }{4}\Big)(\lambda+1)^{\frac{n}{2}-2} (\lambda-(n-3))\\ 
&~~~~~~~~~~~ \times \det \Big(\begin{bmatrix}
        \lambda-\frac{  2\Big((\lambda-b)\frac{n}{2}+(\frac{n^2}{4})\Big)}{\Big((\lambda-b)^2-\frac{n   ^2      }{4}\Big)}& -1-\frac{  2\Big((\lambda-b)\frac{n}{2}+(\frac{n^2}{4})\Big)}{\Big((\lambda-b)^2-\frac{n ^2        }{4}\Big)} \\
        -1-\frac{  2\Big((\lambda-b)\frac{n}{2}+(\frac{n^2 }{4})\Big)}{\Big((\lambda-b)^2-\frac{n         }{2}\Big)} & \lambda-\frac{  2\Big((\lambda-b)\frac{n}{2}+(\frac{n^2}{4})\Big)}{\Big((\lambda-b)^2-\frac{n     ^2    }{4}\Big)}
    \end{bmatrix}-\begin{bmatrix}\frac{n-2}{\lambda-(n-3)} &\frac{n-2}{\lambda-(n-3)} \\
 \frac{n-2}{\lambda-(n-3)}  &\frac{n-2}{\lambda-(n-3)}
    \end{bmatrix}\Big) \\
    &=(\lambda+1)^{\frac{n}{2}-1}
    \big(\lambda^4 + (4-2n)\lambda^3 + (n^2-8n+6) \lambda^2 + (4n^2-14n+4) \lambda +(3n^2-8n+1) \big).
    \end{align*}
Hence, the result follows.
\end{proof}

\begin{theorem}\label{S-CSCom-Q4n}
Let  $Q_{4n} = \langle a, b : a^{2n} = e, a^n =b^2, bab^{-1} = a^{-1}\rangle$ be the dicyclic group of order $4n$. 
\begin{enumerate}
\item If   $n$ is even then $\spec(\mathrm{CSCom}(Q_{4n})) = \{ \, \underbrace{-1,\dots,-1}_{\text{$(4n - 4)${\rm~times}}}, \, n -1, \alpha, \beta, \gamma \,\}$, where $\alpha, \beta, \gamma$ are the roots of the equation $x^3+(3-3n)x^2+(2n^2-10n+3)x+10n^2-15n+1 = 0$.
\item If   $n$  is odd then $\spec(\mathrm{CSCom}(Q_{4n})) = \{ \, \underbrace{-1,\dots,-1}_{\text{$(4n - 4)${\rm~times}}} \, , \alpha, \beta, \gamma, \delta \,\}$, where $\alpha, \beta, \gamma, \delta$ are the roots of the equation $x^4 + (4-4n)x^3 + (4n^2-16n+6)x^2 + (16n^2-28n+4)x +12n^2-16n+1 = 0$.
\end{enumerate} 
\end{theorem}
\begin{proof}
By Theorem \ref{ESCom-D2n}  we have 	$\mathrm{CSCom}(Q_{4n}) \cong \mathrm{CSCom}(D_{2\times 2n})$. Hence, the result follows from Theorem \ref{S-CSCom-D2n}.
\end{proof}

 Notice that the cubic and quartic  equations appearing in Theorems \ref{S-ESCom-D2n} -- \ref{S-CSCom-Q4n} have non-integral roots. For instances, the equations $x^3 -3x^2 -11x + 15 = 0$ has roots $\alpha \approx -2.81114, \beta \approx 1.14307, \gamma \approx 4.66807$; $x^3 -5x^2 -13x + 41 = 0$ has roots $\alpha \approx -3.17226, \beta \approx 2.14399, \gamma \approx 6.02827$; $x^3 -7x^2 + 3x + 31 = 0$ has roots $\alpha \approx -1.72119, \beta \approx 3.35861, \gamma \approx 5.36258$; $x^4 -8x^3 -6x^2+ 64x + 61 = 0$ has roots $\alpha = -1, \beta \approx -2.2034, \gamma = 3.6798, \delta = 7.5236$ etc.

\paragraph{Acknowledgement} 
The first author acknowledges the NFIG grant of the Indian Institute of
Technology Madras RF/22-23/0985/MA/NFIG/009003 and the SERB Startup Research
Grant SRG/2022/001281. The authors are grateful to the referee of an earlier
version for helpful comments.

\end{document}